\documentclass[11pt]{amsart}
\usepackage{amsfonts,amssymb,amscd,amstext,mathrsfs,mathtools}
\usepackage[utf8]{inputenc}
\usepackage{hyperref}
\usepackage{verbatim}

\usepackage{graphics}
\usepackage{graphicx}

\usepackage{times}
\usepackage{enumerate}
\usepackage[up,bf]{caption}
\usepackage{color}
\usepackage{t1enc}

\input xy
\xyoption{all}

\numberwithin{equation}{section}

\newtheorem{theorem}{Theorem}[section] 
\newtheorem{proposition}[theorem]{Proposition} 
 
\newtheorem{lemma}[theorem]{Lemma}

\theoremstyle{definition} 
\newtheorem{definition}[theorem]{Definition} 
\newtheorem{remark}[theorem]{Remark} 
 
\newtheorem{problem}[theorem]{Problem}

\newcommand\Ocal{\mathcal{O}}

\newcommand\Jscr{\mathscr{J}}

\newcommand\C{\mathbb{C}} 
 
\newcommand\CP{\mathbb{CP}}

\newcommand\R{\mathbb{R}}

\newcommand\igot{\mathfrak{i}}

\renewcommand\igot{\mathfrak{i}}

\renewcommand\imath{\igot}

\newcommand\wt{\widetilde} 
\newcommand\wh{\widehat}

\newcommand\Id{\mathrm{Id}}

\def\Ell1{\mathrm{Ell_1}} 
\def\CEll1{\mathrm{CEll_1}}

\begin{document} 

\title{New examples of tame families of Stein manifolds} 

\author{Franc Forstneri{\v c} and Finnur L{\'a}russon}

\address{Franc Forstneri\v c, Faculty of Mathematics and Physics, University of Ljubljana, Jadranska 19, SI--1000 Ljubljana, Slovenia}

\address{Franc Forstneri\v c, Institute of Mathematics, Physics and Mechanics, Jadranska 19, SI--1000 Ljubljana, Slovenia}

\email{franc.forstneric@fmf.uni-lj.si}

\address{Finnur L\'arusson, School of Mathematical Sciences, 
Adelaide University, Adelaide SA 5005, Australia}

\email{finnur.larusson@adelaide.edu.au}

\subjclass[2020]{Primary 32Q56; Secondary 32E10, 32H02}

\date{27 September 2026}

\keywords{Stein manifold, Oka principle, Oka manifold, fibre bundle, vector bundle, principal bundle} 

\begin{abstract}
Let \(X\) be a smooth open manifold of even dimension, \(T\) be a topological space, and \(\{J_t\}_{t\in T}\) be a continuous family of smooth integrable Stein structures on \(X\).  The family is said to be tame if the \(J_t\)-convex hulls of any compact set in \(X\) are upper semicontinuous with respect 
to \(t\in T\). It was recently shown that tame Stein families have good function-theoretic properties; in particular, they satisfy the Oka principle for families of holomorphic maps to any Oka manifold.  We analyse some constructions and operations on Stein manifolds that preserve tameness of families and we exhibit several conceptually new examples of tame Stein families. 
\end{abstract}

\maketitle 


%
%
%

\section{Introduction}\label{sec:intro}

\noindent
In this paper, we study the notion of tameness
of families of Stein manifolds, which was recently 
introduced by the first author and {\'A}lfhei{\dh}ur Edda Sigur{\dh}ard{\'o}ttir
in \cite{ForstnericSigurdardottir2026}.  
Tame Stein families have good function-theoretic properties,
analogous to those of individual Stein manifolds. In particular, 
under suitable regularity assumptions on a parameter space
$T$, every tame family of smooth integrable Stein structures 
$\Jscr=\{J_t\}_{t\in T}$ on a smooth manifold $X$, depending 
continuously on $t\in T$,  
enjoys the Oka--Weil approximation theorem 
for families of holomorphic functions on $X$  
\cite[Theorem 6.3]{ForstnericSigurdardottir2026}, 
as well as the Oka principle with approximation
for families of maps to any Oka manifold
\cite[Theorem 6.1]{ForstnericSigurdardottir2026}.
(See also \cite{Forstneric2026Adv} for results on families
of open Riemann surfaces, in which case every family
of complex structures is tame.)
It is therefore of interest to understand which operations 
on Stein manifolds preserve tameness in families. 

We begin by recalling the definition of a tame Stein family;
see \cite[Definition 5.1]{ForstnericSigurdardottir2026}.

%
%
\begin{definition}\label{def:tame}
A continuous family $\mathscr{J}=\{J_t\}_{t\in T}$ 
of integrable Stein structures on a smooth manifold $X$ is 
{\em tame} at a point $t_0\in T$ if for every compact set $K\subset X$ 
and open set $U\subset X$ containing the 
$J_{t_0}$-convex hull $\wh K_{\! J_{t_0}}$ of $K$ 
there is a neighbourhood $T_0\subset T$ of $t_0$ such that
$\wh K_{\! J_{t}} \subset U$ holds for all $t\in T_0$.
The family $\mathscr{J}$ is tame if it is tame at every point $t_0\in T$. 
\end{definition}

In other words, tameness means that the hulls $\wh K_{J_t}$ 
of any compact set $K$ in $X$ form an 
upper semicontinuous set-valued function of the parameter $t$.  

Assuming that the parameter space $T$ is locally compact Hausdorff, 
\cite[Proposition 5.2]{ForstnericSigurdardottir2026} shows that 
tameness is equivalent to asking that for every compact 
subset $K\subset X$ its $\mathscr{J}$-convex hull 
\begin{equation}\label{eq:whK}
	\wh K_{\!\!\mathscr{J}} =
	\bigcup_{t\in T} \, \{t\}\times \wh{K}_{\! J_t} \subset T\times X
\end{equation}
is such that the projection $\pi:\wh K_{\!\!\mathscr{J}}\to T$ is proper.
Here, $\pi:T\times X\to T$ is the projection on the first factor.
Assuming in addition that the Stein structures $J_t$ are sufficiently 
smooth, tameness of $\Jscr$ is equivalent to each of the following conditions.
\begin{enumerate}[\rm (a)]
\item 
For any compact set $K\subset X$ and $t_0\in T$ 
there is a neighbourhood $T_0\subset T$ such that 
the set $\bigcup_{t\in T_0} \wh K_{J_t}$ is relatively compact in $X$  
(see \cite[Proposition 5.3]{ForstnericSigurdardottir2026}
and note that this holds for an arbitrary topological space $T$).
\item 
For every $t_0\in T$ there exist a neighbourhood $T_0\subset T$
and a continuous function $\rho:T_0\times X\to \R$ 
such that $\rho(t,\cdotp)$ is a strongly $J_t$-plurisubharmonic 
exhaustion function on $X$ for every $t\in T_0$
(see \cite[Theorem 5.5]{ForstnericSigurdardottir2026}).
\item
If $t_0\in T$ and $K\subset X$ is a compact $J_{t_0}$-convex set 
(that is, $K=\wh K_{J_{t_0}}$) 
which is the closure of a strongly $J_{t_0}$-pseudoconvex domain, then 
there is a neighbourhood $T_0\subset T$ of $t_0$ such that $K$ is 
$J_t$-convex for every $t\in T_0$ 
(see \cite[Proposition 5.7]{ForstnericSigurdardottir2026}).
\item
For every $t_0\in T$ and $f\in \Ocal_{J_{t_0}}(X)$ 
there are a neighbourhood $T_0\subset T$ of $t_0$ and a 
$\Jscr$-holomorphic function $F:T_0\times X\to \C$
such that $F(t_0,\cdotp)=f$
(see \cite[Corollary 6.4]{ForstnericSigurdardottir2026}). 
\end{enumerate}

Here, $\Ocal_J(X)$ denotes the space of holomorphic functions
on the complex manifold $(X,J)$.
If $\mathscr{J}=\{J_t\}_{t\in T}$ is a continuous 
family of complex structures on $X$, 
a continuous function $F:T\times X\to\C$ is said to be 
$\Jscr$-holomorphic if the function $F(t,\cdotp):X\to\C$ is 
$J_t$-holomorphic for every $t\in T$.
We denote by $\Ocal_{\!\!\Jscr}(T\times X)$ the space of continuous 
$\Jscr$-holomorphic functions on $T\times X$. 
The analogous notion applies to maps $T\times X\to Y$ to any
complex manifold $Y$, and we use the notation 
$\Ocal_{\!\!\Jscr}(T\times X,Y)$ for the space of continuous 
$\Jscr$-holomorphic maps from $T\times X$ to $Y$.

Every family of complex structures on a smooth surface $X$
is tame, since in this case holomorphic convexity is a 
topological property by the Runge approximation theorem and hence 
independent of the complex structure. Families of open Riemann surfaces 
were treated in \cite{Forstneric2026Adv}. An example of a nontame family 
$\{J_t\}_{t\in\R}$ of smooth Stein structures on $\R^4$, depending 
smoothly on the parameter $t\in\R$, is given in 
\cite[Theorem 4.1]{ForstnericSigurdardottir2026}.
In that example, $(\R^4,J_t)$ is biholomorphic to $\C^2$
for every $t\in \R$ but the $J_t$-convex hull of the closed ball 
explodes as $t\to 0$. The construction is based
on the existence of an injective holomorphic map $F:\C^2\to\C^2$
with non-Runge image $F(\C^2)\subsetneq \C^2$, due to
Wold \cite{Wold2010}. A similar construction applies on bounded 
convex domains $X\subset \C^n$ for $n>1$
\cite[Remark 4.3]{ForstnericSigurdardottir2026}. 

In this paper, we present several constructions 
which preserve tameness of Stein families. 

Recall that a domain 
$\Omega$ in a complex manifold $X$ is said to be {\em Runge} in $X$
if $\{f|_\Omega: f\in \Ocal(X)\}$ is a dense subset of the space
$\Ocal(\Omega)$ of all holomorphic functions on $\Omega$
in the compact-open topology.
For later reference, we recall the following results from  
\cite[Sect.\ 5]{ForstnericSigurdardottir2026}. 
We assume that $\Jscr=\{J_t\}_{t\in T}$ is a continuous 
family of smooth Stein structures on a smooth manifold $X$.

\begin{enumerate}[\rm (i)]
\item
If every point $t_0\in T$ has a neighbourhood $T_0\subset T$
and a family of biholomorphic maps $\Phi_t:(X,J_{t})\to (X,J_{t_0})$ 
depending continuously on $t\in T_0$, then the family $\Jscr$ is tame. 
Indeed, for any compact set $K\subset X$, the sets 
$K_t=\Phi_t(K)$ depend continuously on $t\in T_0$, and hence
their hulls $(\wh {K_t})_{J_{t_0}}$ 
are upper semicontinuous in $t$. Since $\Phi_t$ is 
$(J_{t},J_{t_0})$-biholomorphic, we have 
$
	\wh K_{J_t}= \Phi_t^{-1} \big(  (\wh {K_t})_{J_{t_0}} \big) 
$
for $t\in T_0$ and the claim follows. 
\item
\cite[Proposition 5.9]{ForstnericSigurdardottir2026}:
If for every $t_0\in T$ there are a compact set $L\subset X$ and
a neighbourhood $T_0\subset T$ of $t_0$ such that $J_t=J_{t_0}$
holds on $X\setminus L$ for all $t\in T_0$, 
then $\Jscr$ is tame.
\item
\cite[Proposition 5.11]{ForstnericSigurdardottir2026}:
If $X$ is a smooth manifold, $(Y,J)$ is a Stein manifold, 
and $\Phi_t:X\to \Phi_t(X)\subset Y$ is a continuous family of 
smooth diffeomorphisms onto Stein Runge domains in $Y$, then the 
family of Stein structures $\Phi_t^*J$ on $X$ is tame.
(Note that (i) is a special case of this.)
\item 
\cite[Example 5.12]{ForstnericSigurdardottir2026}: 
Assume that $(Y,J_Y)$ is a Stein manifold 
and $F:T\times X\to Y$ is a continuous map such that 
$F_t=F(t,\cdotp):X \to Y$ is a smooth proper immersion whose image 
$F_t(X)$ is an immersed complex submanifold of $Y$ for each 
$t\in T$. Let $J_t$ denote the unique complex 
structure on $X$ such that the map $F_t$ is $(J_t,J_Y)$-holomorphic.
(We shall say that $J_t=F_t^* J_Y$ is the pullback of $J_Y$ by $F_t$.)
Then, the family $\Jscr=\{J_t\}_{t\in T}$ is tame.
\end{enumerate}

%
%
We now describe the main new results of the present paper. 

In Section \ref{sec:pullbacks}, we establish the tameness of families of 
pullbacks of a Stein fibre bundle by a continuous family of maps from a 
tame Stein family; see Propositions \ref{prop:pullbacks} and 
\ref{prop:Gbundles}.

In Section \ref{sec:Gbundles}, we show that for 
a complex Lie group $G$ with at most finitely many connected components
and a Stein manifold $F$, or for a discrete $F$, a family of holomorphic 
$G$-bundles with fibre $F$ on a tame Stein family $\{(X, J_t)\}_{j\in T}$
has a tame family of Stein total spaces; see Theorem \ref{th:Gbundles}. 
This includes, for example, covering spaces, 
principal bundles, and vector bundles.

In Section \ref{sec:complements}, we show that,
in a tame family of Stein manifolds, the family 
of complements of a continuous family of smooth complex 
hypersurfaces is a tame Stein family; see Theorem \ref{th:complements}. 
The analogous result holds for continuous families of smooth complex 
hypersurfaces in complex projective spaces; see  
Theorem \ref{th:complementsPn}.

We have already mentioned the example of a nontame smooth 
family of Stein structures on Euclidean 4-space $\R^4$, given in 
\cite[Theorem 4.1]{ForstnericSigurdardottir2026}.
We pose the following open problems, the first of which has already been 
mentioned in \cite[Problem 5.13]{ForstnericSigurdardottir2026}.

\begin{problem}
(a) Does there exist a nontame {\em holomorphic} 
family of Stein structures on some smooth manifold $X$?
More precisely, is there a holomorphic submersion $\pi:Z\to T$ 
which is smoothly trivial, with fibre $X$, and whose fibres 
$\pi^{-1}(t)$ determine a nontame family of Stein structures on $X$?

(b) Does every Stein manifold of dimension $>1$ admit a nontame Stein 
deformation? Does this hold for every strongly 
pseudoconvex domain in such a manifold?
\end{problem}

We record the following observation.

\begin{proposition}\label{prop:product}
If $Y$ is a Stein manifold which admits a nontame deformation of its
Stein structure, then the same holds for $X\times Y$ where $X$ is
an arbitrary Stein manifold. In particular, for any Stein manifold 
$X$, the product $X\times \C^2$ admits a nontame Stein deformation.
\end{proposition}

\begin{proof}
If $\{J_{Y,t}\}_{t\in T}$ is a nontame family of Stein structures 
on $Y$ and $(X,J_X)$ is a Stein manifold, then the family of Stein manifolds 
$(X\times Y,J_X\oplus J_{Y,t})$ is obviously nontame. 
\end{proof}

%
%
\begin{remark}
If $X$ is a Stein manifold and $F:X\to X$ is a holomorphic injection which is not surjective, then the image $F(X)\subsetneq X$ is called a Fatou--Bieberbach domain of the second kind in $X$.  Huang, Kutzschebauch, and Rong recently proved that if $Y$ is an algebraically flexible affine algebraic manifold, then $X=Y\times\C$ contains a Fatou--Bieberbach domain of the second kind which is not Runge in $X$ \cite[Theorem 5.9]{HuangKutzschebauchRong2026}.  (This generalises the result of Wold \cite{Wold2010} cited above.)  If $Y$ (or equivalently $X$) is moreover contractible, then the construction in 
\cite[proof of Theorem 4.1]{ForstnericSigurdardottir2026} applies and gives 
a nontame deformation of the Stein structure on $X$ depending smoothly 
on a parameter in $\R$. 

A particularly interesting example in this context is the Koras--Russell 
cubic threefold
\[
	\mathrm{KR}= 
	\left\{(x,y,z_0,z_1)\in \C^4: 
	x^2y+x+z_0^2+z_1^3=0
	\right\}.
\]
Although KR is diffeomorphic to $\R^6$, it is not algebraically isomorphic to
$\C^3$. Makar--Limanov \cite{MakarLimanov1996} proved this by 
introducing and computing the invariant that now bears his name.  
Furthermore, the algebraic automorphism group of KR does not act 
transitively, so KR is not algebraically flexible (see Moser-Jauslin 
\cite{MoserJauslin2011} and Petitjean \cite{Petitjean2016}). On the other 
hand, it is still an open problem whether KR is biholomorphic to $\C^3$.  
Leuenberger proved that KR has the holomorphic density 
property \cite{Leuenberger2016}, so its holomorphic automorphism group 
acts transitively. Poloni has very recently shown that 
$\mathrm{KR}\times\C$ is algebraically flexible \cite{Poloni2026}.  
Hence, \cite[Theorem 5.9]{HuangKutzschebauchRong2026}
implies that $\mathrm{KR}\times \C^2$ admits a nontame deformation 
of its complex structure. However, this is a special case of 
Proposition \ref{prop:product}.

\begin{problem}\label{prob:KR}
Does the Koras--Russell cubic admit a nontame
deformation of its Stein structure?
\end{problem}

The construction of a nontame deformation of a given Stein
structure $J$ on a manifold $X$ in 
\cite[proof of Theorem 4.1]{ForstnericSigurdardottir2026} 
applies more generally if $X$ has a smooth strongly 
plurisubharmonic exhaustion function $\rho:X\to\R$ without critical points 
near infinity and $X$ admits a biholomorphic map $F:X\to F(X)\subsetneq X$ 
whose non-Runge image contains the nontrivial topology of $X$.
This implies that for every sufficiently large $t>0$, there is 
a diffeomorphism $\Psi_t:X\to X$ with $\Psi_t=F$
on $X_t=\{x\in X:\rho(x)\le t\}$, depending smoothly on $t$. 
In this situation, \cite[proof of Theorem 4.1]{ForstnericSigurdardottir2026} 
shows that the continuous family of pullback Stein structures 
$\Psi_t^* J$ on $X$ is nontame.  It is likely that further examples of 
this kind may be obtained from the results in 
\cite{HuangKutzschebauchRong2026}.
\end{remark}

%
%
%
%
\section{Tameness of families of pullbacks}\label{sec:pullbacks}

\noindent
In this section, we prove the following result on tameness of families of 
pullbacks of a holomorphic fibre bundle with Stein total space 
by a continuous family of holomorphic maps from a tame family 
of Stein manifolds. See also Proposition \ref{prop:Gbundles}.

\begin{proposition}\label{prop:pullbacks}
Assume that $\pi:E\to Y$ is a holomorphic fibre bundle whose total space 
$E$ is a Stein manifold, $T$ is a locally compact Hausdorff space, 
$X_t=(X,J_t)$ $(t\in T)$ is a tame family of Stein manifolds, 
and $f_t:X_t\to Y$ is a continuous family of holomorphic maps. 
Assume that the total spaces of the pullback bundles 
$E_t:=f_t^* E\to X_t$ $(t\in T)$ are smoothly diffeomorphic to one 
another, with the diffeomorphisms depending continuously on $t$ 
(in the compact-open topology). Then, the induced family of Stein 
structures on $E_t$ is continuous in $t$ and tame.
\end{proposition}

\begin{proof}
Given a holomorphic map $f:X\to Y$, the pullback of $\pi:E\to Y$ 
has the total space
\[
	f^*E =\{(x,e) \in X\times E: f(x)=\pi(e)\}. 
\]
We have the holomorphic fibre bundle projection $f^*\pi:f^*E\to X$, 
given by $(f^*\pi)(x,e)=x$,
and the holomorphic map $\iota:f^*E\to E$ given by $\iota(x,e)=e$. 
Note that $\iota$ maps the fibre of $f^*E$ over $x\in X$ isomorphically 
onto the fibre of $E$ over $f(x)\in Y$ and the following diagram commutes:
\[
    \xymatrix{ f^* E \ar[d]_{f^*\pi} \ar[r]^{\iota} & E \ar[d]^{\pi} \\
               X \ar[r]^{f}  & Y }
\]
By trivially extending the bundle $\pi:E\to Y$ to 
$\tilde \pi:\wt E\to X\times Y$, 
with $\wt E\cong X\times E$ the pullback of $E$ by the projection 
$X\times Y\to Y$, the bundle $f^*E\to X$ is equivalent to the restriction 
of $\wt E\to X\times Y$ to the graph 
\begin{equation}\label{eq:Gammaf}
	\Gamma_f=\{(x,f(x)):x\in X\} \subset X\times Y,
\end{equation}
with the map $x\mapsto (x,f(x))$ serving as the change of the base.
In this picture, the total space $f^* E=\tilde \pi^{-1}(\Gamma_f)$ 
is a closed complex submanifold of $\wt E$, hence Stein if both 
$X$ and $E$ are Stein. 

Assume now that $X_t=(X,J_t)$ $(t\in T)$, $f_t:X_t\to Y$, and 
$\pi:E\to Y$ are as in the proposition. Note that the family
of Stein structures $\wt J_t$ on $X\times E \cong \wt E$, 
given by $J_t$ on $X$ and by the Stein structure on $E$, is tame.
Since $E_t:=f^*_t E \subset X\times E$ is a continuous family of 
closed $\wt J_t$-complex submanifolds of the tame Stein family 
$(X\times E,\wt J_t)$ for $t\in T$, the family $E_t$ is also tame.
(At this point, we tacitly use the hypothesis that 
the pullback bundles $E_t:=f_t^* E\to X_t$ $(t\in T)$ 
are smoothly diffeomorphic to one another, with a 
continuous dependence of the diffeomorphism on $t$.)
To see this, fix a point $t_0\in T$ and a compact neighbourhood
$T_0\subset T$ of $t_0\in T$ such that there is 
a continuous family of strongly $\wt J_t$-plurisubharmonic 
exhaustion functions $\rho_t:X_t\times E\to\R$ $(t\in T_0)$;
see condition (b) in the introduction 
or \cite[Theorem 5.5]{ForstnericSigurdardottir2026}.
Then the restrictions $\rho_t|_{E_t}:E_t\to\R$ for $t\in T_0$ 
form a continuous family of strongly 
$\wt J_t$-plurisubharmonic exhaustion functions on $E_t$
for $t\in T_0$. Since this holds for every $t_0\in T$, 
the family $\{(E_t,\wt J_t)\}_{t\in T}$ is tame by 
\cite[Theorem 5.5]{ForstnericSigurdardottir2026}.
(A similar argument can be found in 
\cite[Example 5.12]{ForstnericSigurdardottir2026}.)
\end{proof}

The condition on the bundles $E_t$ in Proposition \ref{prop:pullbacks}
to be smoothly diffeomorphic to one another, 
with the diffeomorphisms depending continuously on the parameter $t$,
is not satisfied in general. In particular, non-homotopic maps 
$X\to Y$ may in general give non-diffeomorphic pullbacks.
However, in our applications of this result 
we shall only need the following local case.

%
%
\begin{proposition} \label{prop:Gbundles}
Assume that $E\to Y$ is a smooth fibre bundle, $X$ is a smooth manifold,
$T$ is a topological space, and $f_t:X\to Y$ $(t\in T)$
is a continuous family of smooth maps. Set $E_t:=f_t^* E\to X$ for $t\in T$.
Given a point $t_0\in T$ and a relatively compact domain $U\Subset X$, 
there are a neighbourhood $T_0\subset T$ of $t_0$ and a continuous family 
of smooth fibre bundle diffeomorphisms $\Theta_t: E_{t}|_U\to E_{t_0}|_U$
for $t\in T_0$.
\end{proposition}

\begin{proof}
As before, we consider $E$ as a fibre bundle on $X\times Y$. 
For every $t\in T$, let $F_t:X\to X\times Y$ be the smooth map
\[
	F_t(x)= (x,f_t(x)),\qquad x\in X.
\] 
The pullback $f_t^*E$ is then identified with the restriction of $E$ to 
the closed smooth submanifold 
\[
	\Gamma_{t} = F_t(X) \subset X\times Y,
\]
with $F_t$ serving as the change of the base. 
Let $p:X\times Y\to X$ be the projection $p(x,y)=x$. Fix a point $t_0\in T$. 
By the tubular neighbourhood theorem, $\Gamma_{t_0}=F_{t_0}(X)$ has 
an open tubular neighbourhood $O\subset X \times Y$ with a smooth 
homotopy $h:O\times I \to O$ $(I=[0,1])$ satisfying the following conditions.
\begin{itemize}
\item $p\circ h_s=p$ for all $s\in I$, where $h_s=h(\cdotp,s)$.
\item $h_s|_{\Gamma_{t_0}}$ is the identity on $\Gamma_{t_0}$
for all $s\in I$. 
\item $h_0$ is the identity on $O$. 
\item $h_1(O)=\Gamma_{t_0}$.
\end{itemize}
By the invariance of pullbacks of smooth fibre bundles under
homotopies (see Proposition \ref{prop:invariance}), 
the restricted bundle $E|_O$ is isomorphic to the bundle 
$h_1^* (E|_{\Gamma_{t_0}})$. In particular, we have a smooth
fibre-preserving map $\Theta:E|_O\to E|_{\Gamma_{t_0}}\cong E_{t_0}$, 
covering the retraction $h_1:O\to \Gamma_{t_0}=F_{t_0}(X)$,
which maps any fibre $E_{(x,y)}$ with $(x,y)\in O$ diffeomorphically
onto the fibre $E_{F_{t_0}(x)}$. For any $t\in T$ sufficiently close 
to $t_0$ we have $F_t(U)\subset O$. For such $t$, the restriction of 
$\Theta$ to $E|_{F_t(U)}$ is a fibre bundle diffeomorphism 
\[
	\Theta_t: E|_{F_t(U)} \to E|_{F_{t_0}(U)}
\] 
depending continuously on $t$. Since $E|_{F_t(U)}$ is isomorphic
to $E_{t}|_U$, this completes the proof.
\end{proof}

\begin{remark}\label{rem:constant}
(a) If $T$ is a connected smooth manifold then any two points
$t_0,t\in T$ are connected by a smooth path. By the 
invariance of pullbacks of smooth fibre bundles under
homotopies (see Proposition \ref{prop:invariance}),
the corresponding pullback bundles are isomorphic.
Locally in the parameter $t$ this gives a family 
of isomorphisms $E_t\cong E_{t_0}$ depending continuously on $t$. 
However, it is not clear how to get a well-defined global family of such
isomorphisms for paths in the same homotopy class.

(b) If $X$ is a Stein manifold, $G$ is a complex
Lie group and $\pi:E\to Y$ is a holomorphic $G$-bundle, 
the proof of Proposition \ref{prop:Gbundles},
together with the Oka principle of Grauert \cite{Grauert1958MA}
and the existence of holomorphic tubular neighbourhoods
of Stein submanifolds \cite[Theorem 3.3.3, p.\ 74]{Forstneric2017E},
shows that the family of pullbacks $f_t^* E$ by a continuous family
of holomorphic maps $f_t:X\to Y$ is locally stable (independent of $t$
up to a holomorphic equivalence of holomorphic $G$-bundles)
on any relatively compact Stein domain $U\Subset X$.
A related stability result is due to Leiterer \cite[Theorem 2.7]{Leiterer1990}.
\end{remark}

%
%
%
%
\section{Tameness of holomorphic $G$-bundles on tame Stein families}
\label{sec:Gbundles}

\noindent
In this section, we address the following problem.

\begin{problem}\label{prob:bundles}
Let $X$ be a smooth manifold and $\{J_t\}_{t\in T}$ be a tame family of 
Stein structures on $X$ depending continuously on the parameter $t\in T$. 
Assume that $(E,\wt J_t)\to (X,J_t)$ is a continuous family of holomorphic 
fibre bundles with a given fibre $F$ and structure group $G$, 
and that $(E,\wt J_t)$ is Stein for every $t\in T$.
Under which conditions is the family $\{\wt J_t\}_{t\in T}$ also tame?
\end{problem}

First, we must ask when the manifolds $(E,\wt J_t)$ are Stein. 
An obvious necessary condition is that the fibre $F$ is Stein.  J-P.\ Serre famously asked in 1953 \cite{Serre1953} whether every holomorphic 
fibre bundle with Stein base and Stein fibre has Stein total space:  {\em un espace fibr\'e holomorphe dont la base et la fibre sont des vari\'et\'es de Stein, est-il toujours une vari\'et\'e de Stein?}  The following definition is convenient.

\begin{definition}\label{def:good}
An effective action of a complex Lie group $G$ on a Stein manifold $F$ is {\em good} if, whenever $E$ is a holomorphic $G$-bundle with fibre $F$ over a Stein manifold $X$, the total space $E$ is Stein.  (By an action, we mean a continuous and hence real-analytic action by biholomorphisms.)
\end{definition}

Some of the many affirmative answers to Serre's question may be summarised as follows; we do not attempt a comprehensive survey.

\begin{proposition}\label{prop:good}
The action of a complex Lie group $G$ on a Stein manifold $F$ is 
good if one of the following holds.
\begin{enumerate}[\rm (a)]
\item  $G$ is a linear group (Serre \cite{Serre1953}). This includes the 
case of vector bundles.
\item  $G$ is connected (Matsushima and Morimoto 
\cite[Theorem~6]{MatsushimaMorimoto1960}).
\item  $G$ has finitely many connected components.
\item  $G$ acts freely and transitively on $F$ 
\cite[Theorem~4]{MatsushimaMorimoto1960}.  
This is the case of principal $G$-bundles.
\item  
$F$ is discrete (Stein \cite{Stein1956}).
This is the case of covering spaces, since $G$ can be taken to be the 
monodromy group. This group is countable, being a quotient of the 
fundamental group of $X$, and can therefore be regarded as a 
zero-dimensional Lie group.
\end{enumerate}
\end{proposition}

\begin{proof}
We only need to justify part (c).
Let $P\to X$ be the principal $G$-bundle associated to the  
holomorphic $G$-bundle $E=P\times_G F\to X$, and let $G^0$ denote 
the identity component of $G$. Since $G/G^0$ is finite, the quotient 
$X'=P/G^0 \to X$ is a finite unbranched holomorphic covering, 
so $X'$ is Stein \cite{Stein1956}. The pullback of $E$ to $X'$ 
is naturally isomorphic to $E'=P\times_{G^0}F \to X'$.
Since its structure group $G^0$ is connected, $E'$ is Stein by 
\cite[Theorem 6]{MatsushimaMorimoto1960} (see (b)). 
The natural identification $E'=P\times_{G^0}F \cong X'\times_X E$ 
shows that $E'\to E$ is the pullback of the finite covering $X'\to X$. 
Hence $E'\to E$ is a finite unbranched holomorphic covering.
Since $E'$ is Stein, it follows that $E$ is Stein.
\end{proof}

Several other papers in the literature provide positive answers to Serre's 
question for fibres that are bounded domains of various kinds in Stein 
manifolds and therefore do not admit an effective action by a complex 
Lie group. Such an action is essential for us in order to be able to apply 
Grauert's theorem and get uniquely determined complex structures on 
our total spaces; see Proposition \ref{prop:summary}. 

On the other hand, negative answers to Serre's question with 
infinite-dimensional structure groups have been given by 
Skoda \cite{Skoda1977CR,Skoda1977IM}, 
Demailly \cite{Demailly1978,Demailly1978IM},
Coeure and Loeb \cite{CoeureLoeb1985}, 
Rosay \cite{Rosay2007}, and others.

Let us further recall that if a complex Lie group $G$ acts effectively 
on a Stein manifold, then $G$ is itself Stein 
(see Matsushima and Morimoto \cite[Theorem~3]{MatsushimaMorimoto1960}). 
The question of which complex Lie groups are Stein is answered by the 
following result.

\begin{theorem}[{\cite[Theorems 1 and 2]{MatsushimaMorimoto1960}}]
A complex Lie group $G$ is Stein if and only if $G$
does not contain any connected complex subgroup of positive
dimension contained in a compact subgroup of $G$.
A connected complex Lie group is Stein if and only if 
its connected centre subgroup is isomorphic to the group 
$\C^n\times (\C^*)^m$ for some $n,m\ge 0$.
\end{theorem}

Assume now that $X$ and $F$ are complex manifolds and $G$
is a complex Lie group acting effectively on $F$. 
A fibre bundle $E\to X$ with structure group $G$
is called a {\em $G$-bundle}. If $X$ is a Stein manifold,
Grauert's Oka principle \cite{Grauert1958MA}
says that every topological $G$-bundle $E$
is topologically isomorphic to a holomorphic $G$-bundle, 
and the holomorphic $G$-bundle structure on $E$
compatible with the given topological structure 
is unique up to holomorphic automorphisms of $G$-bundles. 
It follows that the complex structure on $E$, which makes it a 
holomorphic $G$-bundle in a given topological isomorphism class of 
$G$-bundles, is uniquely determined. 
We summarise these observations as follows.

%
%
\begin{proposition}\label{prop:summary}
Let $X$ be a smooth manifold, $F$ be a complex manifold,
$G$ be a complex Lie group acting effectively on $F$, and $E\to X$
be a smooth $G$-bundle with fibre $F$.
A family $\{J_t\}_{t \in T}$ of smooth 
Stein structures on $X$ determines on $E\to X$ a family
of $J_t$-holomorphic $G$-bundle structures,
which are unique up to holomorphic isomorphisms. 
Hence, there is a unique family $\{\wt J_t\}_{t\in T}$
of associated complex structures on $E$ determined by these conditions.
If $J_t$ is continuous in $t\in T$, then so is $\wt J_t$.
If $F$ is Stein and the action of $G$ on $F$ is good, 
then $(E,\wt J_t)$ is a Stein manifold for every $t\in T$.
\end{proposition}

The statement in the above proposition regarding continuity of the 
complex structures $\wt J_t$ with respect to the parameter $t\in T$ 
follows from Remark \ref{rem:continuous}.

The following is the main result of this section.

%
%
\begin{theorem}\label{th:Gbundles}
Assume that $X$ is a smooth manifold and 
$\pi:E\to X$ is a $G$-bundle whose fibre $F$ is a Stein manifold 
with a good action (see Definition \ref{def:good}) of the structure group $G$. 
Let $T$ be a locally compact Hausdorff space and $\{J_t\}_{t\in T}$ be a 
tame family of smooth Stein structures on $X$. Then the associated family 
$\{\wt J_t\}_{t\in T}$ of Stein structures on $E$ (see 
Proposition \ref{prop:summary}) is also tame.
\end{theorem}

This result applies in all cases mentioned in Proposition \ref{prop:good}.

%
%
\begin{proof} 
We first show that for any $t_0\in T$ and smoothly bounded 
strongly $J_{t_0}$-pseudoconvex domain $U\Subset X$, 
the family of Stein structures $\wt J_t$ on the 
restricted bundle $E|_U$ is continuous and tame at $t_0$. 

Set $E_t=(E,\wt J_t)$ for $t\in T$. 
Fix a smoothly bounded, strongly $J_{t_0}$-pseudoconvex 
domain $\Omega \Subset X$ with $\overline U\subset \Omega$. 
By \cite[Theorem 3.1]{ForstnericSigurdardottir2026} there are a
neighbourhood $T_0\subset T$ of $t_0$ and a family 
of $(J_t,J_{t_0})$-biholomorphisms 
\begin{equation}\label{eq:Phit}
	\Phi_t : \Omega \to \Phi_t(\Omega) \subset X\quad (t\in T_0),
	\quad \Phi_{t_0} = \Id_{\Omega},
\end{equation}
depending continuously on $t\in T_0$. 
For every $t\in T_0$, the pullback  
$(\Phi^*_t E_{t_0})|_\Omega \to \Omega$ is a $J_t$-holomorphic 
$G$-bundle with fibre $F$. Assuming that 
the neighbourhood $T_0$ of $t_0$ is small enough, the restriction 
$(\Phi^*_t E_{t_0})|_U$ is isomorphic to $E_t|_U$ 
as a topological $G$-bundle (see Proposition \ref{prop:Gbundles}).
By Grauert's Oka principle \cite{Grauert1958MA}, 
$(\Phi^*_t E_{t_0})|_U$ is isomorphic to $E_t|_U$ as a 
$J_t$-holomorphic $G$-bundle. 
Under this holomorphic bundle isomorphism, the complex structure $\wt J_t$ 
on $E_t|_U$ corresponds to the complex structure on the pullback bundle 
$(\Phi_t^*E_{t_0})|_U$. Equivalently, it is obtained from $\wt J_{t_0}$ by a 
bundle diffeomorphism covering the map 
$\Phi_t|_U:U\to\Phi_t(U) \subset X$ (see \eqref{eq:Phit}). 
Hence, Propositions \ref{prop:pullbacks} and \ref{prop:Gbundles}
imply that the family of complex structures $\{\wt J_t\}_{t\in T_0}$ on 
$E|_U$ is continuous and tame if the neighbourhood $T_0$ 
is small enough. Since this argument applies at every point $t_0\in T$, 
we conclude that the family of complex structures 
$\{\wt J_t\}_{t\in T}$ on $E$ is continuous.

%
%
\begin{remark}\label{rem:continuous}
So far we have not used the hypothesis that the action of the structure 
group $G$ on the fibre $F$ is good. This shows that if the family of complex
structures $\{J_t\}_{t\in T}$ on $X$ is continuous in $t$, then the
associated family $\{\wt J_t\}_{t\in T}$ of complex structures on $E$
is also continuous in $t$. 
\end{remark}
 
To complete the proof, it remains to show that the family of Stein 
structures $\wt J_t$ on $E$ is tame. We need the following observation.

%
%
\begin{proposition}\label{prop:Runge} 
Let $p:Y \to X$ be a holomorphic map between Stein manifolds 
and let $U \subset X$ be a Runge open subset. 
Assume that $D:= p^{-1}(U)$ is Stein. Then $D$ is Runge in $Y$. 
\end{proposition} 

Recall that a Stein domain $D$ in a Stein manifold $Y$ is Runge 
if and only if it admits an exhaustion by compact 
$\mathcal O(Y)$-convex sets \cite[Theorem 4.3.3]{Hormander1990}.
The proposition follows immediately from the following lemma
by choosing exhaustions of $U$ by compact $\mathcal O(X)$-convex sets 
$L_j$ and of $Y$ by compact $\mathcal O(Y)$-convex sets $B_j$. 

\begin{lemma} \label{lema:Runge}
Let $p:Y\to X$ be a holomorphic map between Stein manifolds. 
If $L\subset X$ is compact and $\mathcal O(X)$-convex 
and $B\subset Y$ is compact and $\mathcal O(Y)$-convex, then 
$K:=p^{-1}(L)\cap B$ is $\mathcal O(Y)$-convex. 
\end{lemma}

\begin{proof} 
Let $y\in Y\setminus K$. Suppose that $p(y)\notin L$. 
Since $L$ is $\mathcal O(X)$-convex, 
there exists $f\in\mathcal O(X)$ such that 
$
	|f(p(y))| > \sup_{x\in L}|f(x)|. 
$ 
The pullback $f\circ p$ belongs to $\mathcal O(Y)$, and 
\[ 
	|(f\circ p)(y)| > \sup_{z\in K}|(f\circ p)(z)| 
\] 
because $p(K)\subset L$. Hence $y$ does not belong to the 
$\mathcal O(Y)$-hull of $K$. Suppose now that $p(y)\in L$. 
Since $y\notin K$, necessarily $y\notin B$. As $B$ is 
$\mathcal O(Y)$-convex, there exists $g\in\mathcal O(Y)$ satisfying 
\[ 
	|g(y)| > \sup_{z\in B}|g(z)| \ge \sup_{z\in K}|g(z)|. 
\] 
Hence again $y$ does not belong to the $\mathcal O(Y)$-hull of $K$. 
\end{proof}

We now complete the proof of Theorem \ref{th:Gbundles}. Fix $t_0\in T$
and a compact $\Ocal(E,\wt J_{t_0})$-convex set $L\subset E$.
Given an open set $V\subset E$ with $L\subset V$, we must find 
a neighbourhood $T_0\subset T$ of $t_0$ such that 
\begin{equation}\label{eq:inclusionL}
	\wh L_{(E,\wt J_t)}\subset V \ \ \text{for all $t\in T_0$}.
\end{equation}
Choose an open smoothly bounded domain $U\Subset X$ which is 
strongly pseudoconvex and Runge in $(X,J_{t_0})$ such 
that $L \subset E|_U$. Since the family $\{J_t\}_{t\in T}$ 
is tame, there is a neighbourhood $T_0\subset T$ of $t_0$ 
such that $U$ is strongly $J_t$-pseudoconvex and 
Runge in $(X,J_t)$ for all $t\in T_0$ 
\cite[Proposition 5.7]{ForstnericSigurdardottir2026}.
By the first part of the proof, the family $(E|_U,\wt J_t)$
is tame at $t_0$. Hence, shrinking $T_0$ if necessary, we have  
\[
	\wh L_{(E|_U,\wt J_t)}\subset V\ \ \text{for all $t\in T_0$}.
\]
Since $E|_U$ is $\wt J_t$-Runge in $E$ for all $t\in T_0$
by Proposition \ref{prop:Runge}, the $\wt J$-hull of $L$ in
$E|_U$ agrees with its $\wt J_t$-hull in $E$. This shows that  
\eqref{eq:inclusionL} holds, thereby completing the proof.
\end{proof}

%
%
\section{Complements of families of complex hypersurfaces}
\label{sec:complements}

\noindent
If $X$ is a Stein manifold and $H$ is a closed complex hypersurface in $X$,
then the complement  $X\setminus H$ is also a Stein manifold 
(see \cite[Theorem 5, p.\ 129]{GrauertRemmert1979}).
Suppose now that $T$ is a topological space and 
$\mathscr{J}=\{J_t\}_{t\in T}$ is a continuous family of smooth 
Stein structures on $X$.

\begin{definition}\label{def:continuous}
A family $H_t\subset X$ $(t\in T)$ of closed $J_t$-complex hypersurfaces
is said to be continuous if for some (hence for any)
$t_0\in T$ there exists a continuous family of smooth diffeomorphisms 
$\Theta_t:X\to X$ $(t\in T)$ such that $\Theta_{t_0}=\Id_X$ and 
$H_t=\Theta_t(H_{t_0})$ for every $t\in T$. 
\end{definition}

Note that $\wt J_t=\Theta_t^* J_t$ for $t\in T$ is then a continuous 
family of Stein structures on the smooth manifold $X$ 
such that $H=H_{t_0}$ is a $\wt J_t$-complex hypersurface in 
$X$ for every $t\in T$. 

\begin{theorem}\label{th:complements}
(Notation as above.) 
If the family of Stein structures $\Jscr=\{J_t\}_{t\in T}$ on $X$ is tame, 
then the family of Stein structures $\wt \Jscr=\{\wt J_t\}_{t\in T}$ on $X\setminus H$ is also tame.
\end{theorem}

The conclusion of Theorem \ref{th:complements} is equivalent to the 
statement that the family of Stein manifolds 
$\{(X\setminus H_t,J_t)\}_{t\in T}$ is tame. 
This viewpoint will be convenient in the proof.
However, since the notion of tameness formally pertains to a family of 
Stein structures on a fixed smooth manifold, we have transferred the 
Stein structures $J_t$ on the variable family of domains
$X\setminus H_t$ to the Stein structures $\wt J_t=\Theta_t^* J_t$
on the fixed domain $X\setminus H$ via the diffeomorphisms $\Theta_t$.
Here is an equivalent statement.

\begin{theorem}\label{th:complementsbis}
If $\{J_t\}_{t\in T}$ is a tame family of smooth Stein structures on 
$X$  and $H\subset X$ is a closed $J_t$-holomorphic hypersurface 
for every $t\in T$, then the Stein family 
$\{(X\setminus H,J_t)\}_{t\in T}$ is also tame.
\end{theorem}

In the proof of Theorem \ref{th:complements} 
we shall use the following observation.

%
%
\begin{lemma}\label{lem:linebundle}
Let $\pi:E\to X$ be a holomorphic line bundle on a Stein manifold $X$, 
and let $f_t:X\to E$ $(t\in T)$ be a continuous family of holomorphic 
sections. Set $H_t=f_t(X)$. Then there exists a continuous family 
$\phi_t:E\setminus H_t\to \R_+$ $(t\in T)$ of smooth strongly plurisubharmonic 
exhaustion functions. In particular, $\{E\setminus H_t\}_{t\in T}$ 
is a tame family of Stein manifolds.
\end{lemma}

\begin{proof}
Since $E$ is Stein, the complement $E\setminus E_0$ 
of its zero section $E_0$ is also Stein 
(see \cite[Theorem 5, p.\ 129]{GrauertRemmert1979}). 
Hence, there is a smooth strongly plurisubharmonic exhaustion function 
$\phi:E\setminus E_0\to\R_+$. We identify $E_0$ with $X$.
For every $t\in T$, the map $F_t:E\to E$, defined by 
\[
	F_t(e)=e-f_t(\pi(e)), \qquad e\in E, 
\]
is a fibre-preserving biholomorphism of $E$ which is affine 
linear on each fibre and maps $H_t=f_t(X)$ onto $E_0\cong X$. 
The family $\phi_t=\phi\circ F_t$ satisfies the lemma.
The biholomorphisms $F_t:E\setminus H_t\to E\setminus E_0$ 
identify the varying complements $\{E\setminus H_t\}_{t\in T}$ 
with the fixed smooth manifold $E\setminus E_0$. 
\end{proof} 

We shall need the following version of 
Rossi's local maximum principle \cite{Rossi1960}, 
due to Rosay \cite[Proposition, p.\ 430]{Rosay2006}.
It is stated in terms of plurisubharmonic functions as opposed
to holomorphic functions. Recall that in any Stein manifold, 
the plurisubharmonic hull of a compact set agrees with its holomorphic 
hull (see \cite[Theorem 5.2.10]{Hormander1990}).  
We state Rosay's result for complex manifolds, although his 
result also applies to almost complex (not necessarily integrable) 
manifolds.

%
%
\begin{proposition}\label{prop:Rosay}
Let $K$ be a compact set in a complex manifold $X$, and let
$\wh K$ be its plurisubharmonic hull. 
Let $x \in  \wh K \setminus K$. For any relatively compact 
neighbourhood $V\Subset X$ of $x$ that does not intersect $K$ 
and any continuous plurisubharmonic function
$u$ defined on a neighbourhood of $\overline V$, we have 
\[
	u(x)\le \sup \big\{u(y): y\in \wh K \cap bV\big\}.
\]
\end{proposition}

\begin{proof}[Proof of Theorem \ref{th:complements}] 
Using the notation in Definition \ref{def:continuous},
it suffices to prove tameness of the family $\wt \Jscr$ at $t=t_0$, for
the situation is symmetric with respect to the choice of $t_0\in T$.

Fix a compact set $K\subset X\setminus H_{t_0}$. Since the 
family $H_t$ depends continuously on $t$, 
there is a neighbourhood $T_0\subset T$ of $t_0$ such that 
$K\subset X\setminus H_{t}$ for all $t\in T_0$.
Pick a smooth strongly $J_{t_0}$-plurisubharmonic exhaustion function 
$\rho:X\to \R$ such that the domain $U=\{\rho<0\}$ is smoothly bounded 
and contains $\wh K_{J_{t_0}}$. Since the family of Stein structures
$\Jscr=\{J_t\}_{t\in T}$ is tame, we can shrink the neighbourhood 
$T_0\subset T$ around $t_0$ to ensure that the set 
$\bigcup_{t\in T_0}\wh K_{J_t}$ is relatively compact in $U$. 
Let $K_t$ denote the $J_t$-convex hull of $K$ in the Stein manifold 
$(X\setminus H_t,J_t)$ for $t\in T_0$. Note that 
\[ 
	K_t \subset \wh K_{J_t} \setminus H_t \subset U
	\ \ \text{for $t\in T_0$}.
\] 
To complete the proof of the theorem, we show the following.

\begin{lemma}\label{lem:pshfunctions}
(Notation as above.)  After shrinking $T_0$ around $t_0$,
there is a continuous family of strongly 
$J_t$-plurisubharmonic exhaustion functions 
$\phi_t:U\setminus H_t\to\R$ for $t\in T_0$.
\end{lemma}

From Lemma \ref{lem:pshfunctions} and Proposition \ref{prop:Rosay} 
it follows that the hulls $K_t$ of $K$ in $(X\setminus H_t,J_t)$ 
are bounded away from $H_t\cup bU$ for $t\in T_0$. 
Since the point $t_0\in T$ and the compact set 
$K\subset X\setminus H_{t_0}$ were arbitrary, property (a)
in the introduction (see \cite[Proposition 5.3]{ForstnericSigurdardottir2026}) 
implies that the family $\{(X\setminus H_t,J_t)\}_{t\in T}$ is tame.
This is equivalent to saying that the family 
$\{(X\setminus H_{t_0},\wt J_t)\}_{t\in T}$ is tame.

\begin{proof}[Proof of Lemma \ref{lem:pshfunctions}]
Let $\rho:X\to\R$ be as above with $U=\{\rho<0\}$.
If $T_0$ is a small enough neighbourhood of $t_0$ then $\rho|_U$ is 
strongly $J_t$-plurisubharmonic for every $t\in T_0$
(see \cite[Lemma 5.6]{ForstnericSigurdardottir2026}).
Pick a smooth, convex, strongly increasing function 
$\xi:(-\infty,0)\to\R$ with $\lim_{t\to 0}\xi(t)=+\infty$. 
It follows that $\xi \circ \rho:U\to \R$ is a strongly $J_t$-plurisubharmonic 
exhaustion function for every $t\in T_0$. 
This is an immediate consequence of the formulas
\begin{equation}\label{eq:ddc}
	dd^c(\xi \circ \rho) = 
	(\xi'\circ\rho) dd^c\rho + (\xi''\circ \rho) d\rho\wedge d^c\rho
\end{equation}
where $\xi'$ is the derivative of $\xi$, 
$d^c\rho (Jv) = d\rho(-Jv)$ for $v\in TX$, 
$dd^c\rho$ is the Levi form of $\rho$, and 
\begin{equation}\label{eq:dwedgedc}
	(d\rho\wedge d^c\rho)(v \wedge Jv)=
	(d\rho(v))^2 + (d^c\rho(v))^2,
\end{equation}
which hold for any complex structure $J$ on $X$ and tangent vector 
$v \in TX$. (This also follows from the formula 
\cite[Eq.\ (8.9), p.\ 343]{AlarconForstnericLopez2021} for the
Hessian of $\xi\circ\rho$, noting that $dd^c\rho$ on a complex
line $L=\mathrm{Span}(v,Jv)\subset TX$ $(0\ne v\in TX)$
is the trace of the Hessian of $\rho$ restricted to $L$.) 

We shall find a continuous family of strongly 
$J_t$-plurisubharmonic functions $\tau_t$ on a neighbourhood of 
$H_t\cap \overline U$ in $X$ such that $\tau_t$ tends to $+\infty$ 
along the hypersurface $H_t$ for every $t\in T_0$. 
For a suitable constant $c>0$ and up to shrinking $T_0$ around $t_0$, 
\[
	\phi_t := \max\{ \xi \circ \rho, c \, \tau_t\}
\]
is then a well-defined strongly $J_t$-plurisubharmonic exhaustion 
function on $U\setminus H_t$ depending continuously on $t\in T_0$. 
To make $\phi_t$ smooth, we replace $\max$ by the regularised 
maximum, see \cite[p.\ 69]{Forstneric2017E}. However, this is 
inessential for the application.

It remains to construct the functions $\tau_t$ with the stated properties.
In the sequel, we will shrink $T_0$ around $t_0$ when necessary 
without mentioning it again. 
By the Docquier--Grauert tubular neighbourhood theorem
(see \cite[p.\ 257, Theorem 8]{GunningRossi1965} or
\cite[Theorem 3.3.3, p.\ 74]{Forstneric2017E}), the hypersurface 
$H_{t_0}$ has an open $J_{t_0}$-Stein neighbourhood $V \subset X$
which is $J_{t_0}$-biholomorphic to a neighbourhood $\wt V\subset E$
of the zero section $E_0$ of $E$. 
Denote by $\Psi:V\to \wt V$ this biholomorphism and identify
$H_{t_0}$ with the zero section $\Psi(H_{t_0})=E_0$ of $E$. 
Choose a pair of strongly $J_{t_0}$-pseudoconvex domains 
$\Omega \Subset \Omega_0 \Subset X$ with $\overline U\subset \Omega$. 
By \cite[Theorem 3.1]{ForstnericSigurdardottir2026}
there is a continuous family of $(J_t,J_{t_0})$-biholomorphic 
maps $\Phi_t:\Omega_0\to \Phi_t(\Omega_0)\subset X$, with 
$\Phi_{t_0}=\Id_{\Omega_0}$, such that $\Omega\subset\Phi_t(\Omega_0)$ 
for every $t\in T_0$. Let $\pi:E\to H_{t_0}$ denote the 
$J_{t_0}$-holomorphic normal line bundle of $H_{t_0}$ in $X$. 
Recall that $H_t$ is a $J_t$-complex hypersurface in $X$ for all $t\in T$.
For $t$ close to $t_0$ the set 
\[	
	\wt H_t:=\Phi_t(H_t\cap \Omega_0) \subset X 
\]
is then a locally closed $J_{t_0}$-complex hypersurface 
close to $H_{t_0}$ and depending continuously on $t$. 
In particular, we have $\wt H_t \subset V$ and 
the image $\Psi(\wt H_t)\subset \wt V\subset E$ 
coincides over the $J_{t_0}$-Stein domain $W:=E_0 \cap \Omega$ 
with the image of a $J_{t_0}$-holomorphic section 
$f_t: W\to E|_W$. (Recall that $E_0$ is identified with $H_{t_0}$.)
Clearly, $f_t$ depends continuously on $t$. 
By Lemma \ref{lem:linebundle}, there is a continuous
family of strongly $J_{t_0}$-plurisubharmonic exhaustion functions
$\psi_t: E|_W \setminus f_t(W)\to\R$, $t\in T_0$. Taking
\[
	\tau_t = \psi_t \circ \Psi\circ \Phi_t, \quad t\in T_0,
\]
gives a continuous family of strongly $J_t$-plurisubharmonic
functions on deleted open neighbourhoods of $H_t\cap \Omega$
which tend to $+\infty$ along $H_t\cap \Omega$. 
This completes the proof.
\end{proof}

As noted above, this proves Theorem \ref{th:complements}.
\end{proof}

We have the following analogue of Theorem \ref{th:complements} for 
hypersurfaces in complex projective spaces.

\begin{theorem}\label{th:complementsPn}
Let $n>1$ and $T$ be a topological space. 
If $H_t\subset \CP^n$ $(t\in T)$ is a continuous family of smooth 
complex hypersurfaces, then the complements 
$X_t=\CP^n\setminus H_t$ $(t\in T)$,  
with the induced complex structures from $\CP^n$,  
form a tame family of Stein manifolds. 
\end{theorem}

\begin{proof}[Sketch of proof]
An inspection of the construction in 
\cite[proof of Theorem 7.2]{DrinovecForstneric2010AJM},
taking into account that every closed complex submanifold 
of $\CP^n$ has positive normal bundle \cite{Barth1970},
shows that there is a family of open tubular neighbourhoods 
$U_t\subset \CP^n$ of $H_t$ and smooth strongly plurisubharmonic 
functions $\rho_t : U_t\setminus H_t\to\R$, depending 
continuously on $t\in T$, such that $\rho_t$ tends to $+\infty$
along $H_t$ for every $t\in T$. It follows that 
$X_t =\CP^n\setminus H_t$ is a holomorphically convex 
manifold. Since any closed complex subvariety of $\CP^n$
of positive dimension intersects $H_t$, $X_t$
does not contain any such subvariety and hence it is Stein.
By our definition of a continuous family of hypersurfaces, 
the Stein manifolds $X_t$ are smoothly diffeomorphic to one another
with a continuous dependence of the diffeomorphism on $t$.
The existence of functions $\rho_t$ as above shows that for any $t_0\in T$
and compact set $K\subset X_{t_0}$, the $\Ocal(X_t)$-convex
hull $\wh K_t$ of $K$ for $t$ near $t_0$ does not enter the region 
in $U_t$ where $\rho_t$ is large positive; in particular, 
it does not approach $H_t$. This shows tameness.

An alternative argument is to show that the functions 
$\rho_t$ extend from smaller deleted neighbourhoods
$U_t'\setminus H_t$ of $H_t$ to strongly plurisubharmonic
exhaustion functions $\tilde\rho_t:X_t\to\R$
depending continuously on $t$ near $t_0$.
The conclusion then follows from condition (b) in the Introduction.
\end{proof}

%
%
%
%
\section{Appendix: Homotopy invariance of pullbacks}

%
%
\noindent
The following result is well known. 

\begin{proposition}\label{prop:invariance}
Let $\pi:E \to Y$ be a smooth fibre bundle, let $X$ be a smooth manifold,
and let $f,g:X\to Y$ be smoothly homotopic maps. Then the pullback bundles
$f^*E$ and $g^*E$ are isomorphic as smooth fibre bundles over $X$.
\end{proposition}

For principal bundles, this is proved, for example, in Steenrod's book
\cite[Part II, Chapter 11]{Steenrod1951}. The general case can be found 
in several sources. However, as pointed out by del Hoyo \cite{delHoyo2016},
most of these proofs contain a gap, and he explained how to remove it. 
We include the main point of his argument.

\begin{proof}
Set $I=[0,1]$ and let $H:X\times I \to Y$ be a smooth homotopy 
with $H(\cdotp,0)=f$ and $H(\cdotp,1)=g$. Consider the pullback bundle
$H^*E \to X\times I$. Its restriction to $X\times\{0\}$ is naturally 
identified with $f^*E$, while its restriction to $X\times\{1\}$ is naturally identified 
with $g^*E$. Choose a complete Ehresmann connection on the smooth fibre 
bundle $H^*E \to X\times I$ (see del Hoyo \cite{delHoyo2016}
for the existence of such a connection). Let $t$ be the standard coordinate on 
$I$ and $\partial_t=\frac{\partial}{\partial t}$ be the corresponding vector 
field on $X\times I$. Its horizontal lift with respect to the complete
Ehresmann connection is a complete vector field on
$H^*E$ covering $\partial_t$. Integrating this vector field from $t=0$
to $t=1$ yields a diffeomorphism
\[
	\Phi : (H^*E)|_{X\times\{0\}}
 	  \longrightarrow
	   (H^*E)|_{X\times\{1\}}
\]
which preserves fibres. Using the above identifications of the end fibres with 
the pullback bundles, we obtain a smooth bundle isomorphism
$f^*E \cong g^*E$. 
\end{proof}

The above proof fails in general if the Ehresmann connection
on $H^*E$ is not complete, since the integral curves 
of the lifted vector field may escape to infinity before
the parameter $t\in I$ reaches the value $t=1$.
In particular, the connection obtained
as the orthogonal complement to the vertical subbundle of the 
fibre bundle projection with respect to some Riemannian metric
on $H^*E$ need not be complete. It was proved in  
\cite{delHoyo2016} that a smooth surjective submersion $\pi:E\to X$ 
admits a complete Ehresmann connection if and only if it is a fibre bundle.

%
%
%
%
\medskip 
\noindent {\bf Acknowledgements.} 
Forstneri\v c was supported by the European Union (ERC Advanced grant HPDR, 101053085) and grants P1-0291 and N1-0237 from ARIS, 
Republic of Slovenia.  Part of this work was done during Forstneri\v c's 
visit to Adelaide University in February 2026, and he wishes to thank 
the university for its hospitality. 

\smallskip

\noindent  {\bf Use of artificial intelligence.}  
AI assistants were used to search for references and weed out minor mistakes, inconsistent notation, and typographical errors. 




\end{document}